\documentclass[11pt]{amsart}

\newtheorem{theorem}{Theorem}

\newtheorem{corollary}[theorem]{Corollary}

\newtheorem{lemma}[theorem]{Lemma}

\newtheorem{problem}[theorem]{Problem}
\newtheorem{proposition}[theorem]{Proposition}
\newtheorem{remark}[theorem]{Remark}

\usepackage[active]{srcltx} 

\def\J#1#2#3{ \left\{ #1,#2,#3 \right\} }

\def\11{\textbf{$1$}}

\usepackage{enumerate}
\usepackage{amsmath}
\usepackage{amssymb}

\begin{document}

\title[Local triple derivations on C$^*$-algebras]{Local triple derivations on C$^*$-algebras}

\author[Burgos]{Maria Burgos}
\email{maria.burgos@uca.es}
\address{Departamento de Matematicas, Facultad de Ciencias Sociales y de la Educacion,  Universidad de Cadiz, 11405, Jerez de la Frontera, Spain.}

\author[Fern\'{a}ndez-Polo]{Francisco J. Fern\'{a}ndez-Polo}
\email{pacopolo@ugr.es}
\address{Departamento de An{\'a}lisis Matem{\'a}tico, Facultad de
Ciencias, Universidad de Granada, 18071 Granada, Spain.}

\author[Garc{\' e}s]{Jorge J. Garc{\' e}s}
\email{jgarces@correo.ugr.es}
\address{Departamento de An{\'a}lisis Matem{\'a}tico, Facultad de
Ciencias, Universidad de Granada, 18071 Granada, Spain.}

\author[Peralta]{Antonio M. Peralta}
\email{aperalta@ugr.es}
\address{Departamento de An{\'a}lisis Matem{\'a}tico, Facultad de
Ciencias, Universidad de Granada, 18071 Granada, Spain.}

\thanks{Authors partially supported by the Spanish Ministry of Economy and Competitiveness,
D.G.I. project no. MTM2011-23843, and Junta de Andaluc\'{\i}a grants FQM0199 and
FQM3737. The first author was partially supported by a Grant from the GENIL-YTR program.}


\date{}

\begin{abstract} We prove that every bounded local triple derivation on a unital C$^*$-algebra
is a triple derivation. A similar statement is established\hyphenation{established} in the category of unital JB$^*$-algebras.
\end{abstract}


\maketitle
 \thispagestyle{empty}

\section{Introduction}

In a pioneering work, R. Kadison started, in 1990, the study of local derivations from an associative algebra $\mathcal{R}$ into an $\mathcal{R}$-bimodule $\mathcal{M}$ (cf. \cite{Kad90}). We recall that a
linear mapping $D : \mathcal{R} \to \mathcal{M}$ is a \emph{derivation} or an \emph{associative derivation}
whenever $D(a b) = D(a) b + a D(b),$ for every $a,b\in \mathcal{R}.$ In words of Kadison ``The defining property of a linear mapping $T : \mathcal{R} \to \mathcal{M}$ to be a \emph{local (associative) derivation} is that for
each $a$ in $\mathcal{R}$, there is a derivation $D_{a} : \mathcal{R} \to \mathcal{M}$ such that $T(a) =D_{a}(a)$''. R. Kadison proved that each
norm-continuous local derivation of a von Neumann algebra $W$ into a dual $W$-bimodule is a derivation (cf. \cite[Theorem A]{Kad90}).
B.E. Johnson extended the above result proving that every (continuous) local derivations from any C$^*$-algebra $B$ into any Banach $B$-bimodule is a derivation (see \cite[Theorem 5.3]{John01}). Concerning the hypothesis of continuity in the above result, let us briefly notice that J.R. Ringrose proved that every (associative) derivation from a C$^*$-algebra $B$ to a Banach $B$-bimodule $X$ is continuous (cf. \cite{Ringrose72}). In \cite{John01}, B.E. Johnson also gave an automatic continuity result, showing that local derivations on C$^*$-algebras are continuous even if not assumed a priori to be so.\smallskip

The above results motivated a multitude of studies on local derivations on C$^*$-algebras. There exists a rich list of references revisiting, rediscovering and extending Kadison-Johnson theorem in many directions (see, for example, \cite{Crist,HadLi,HadLi08,Hv,LarSo,LiPan,LiPanXu,Shu} and \cite{ZhanPanYang}).\smallskip

C$^*$-algebras belong to a more general class of Banach spaces, called JB$^*$-triples, which is defined in terms of algebraic, topological and geometric axioms mutually interplaying (see section 2 for more details). Originally introduced by W. Kaup in the classification of \emph{bounded symmetric domains} on arbitrary complex Banach spaces (cf. \cite{Ka}), JB$^*$-triples now have their own importance in Functional Analysis and Geometry of Banach spaces. A \emph{triple} or \emph{ternary} derivation on a JB$^*$-triple $E$ is a linear mapping $\delta : E \to E$ satisfying: \begin{equation}\label{eq derivation} \delta \J abc = \J {\delta (a)}bc + \J a{\delta (b)}c+ \J ab{\delta (c)},
\end{equation} for every $a,b,c\in E$. In the setting of JB$^*$-triples, J.T. Barton and Y. Friedman proved that every triple derivation on a JB$^*$-triple is continuous (cf. \cite{BarFri}, see also \cite{HoMarPeRu}). A \emph{local triple derivation} on $E$ is a linear map $T : E\to E$ such that for each $a$ in $E$ there exists a triple derivation $\delta_{a}$ on $E$ satisfying $T(a) = \delta_a (a).$\smallskip

Quite recently, Jordan Banach triple modules over a JB$^*$-triple and triple derivations from a JB$^*$-triple $E$ to a Jordan Banach triple $E$-module $X$ were introduced by B. Russo and the fourth author of this note. In \cite{PeRu} these authors provide necessary and sufficient conditions under which a derivation from $E$ into $X$ is continuous. We refer to \cite{PeRu} and \cite{HoPeRu} for the basic definitions and results on JB$^*$-triples, Jordan Banach triple modules and triple derivations not included in this note. Following \cite{PeRu}, a conjugate linear mapping $\delta : E \to X$ is a \emph{triple or ternary derivation} whenever it satisfies the above identity $(\ref{eq derivation})$. In particular, the dual, $E^*$, of a JB$^*$-triple $E$, is a Jordan Banach triple $E$-module and every triple derivation from $E$ into $E^*$ is continuous (see \cite[Corollary 15]{PeRu}). Furthermore, every triple derivation from a
C$^*$-algebra $B$ to a Banach triple $B$-module is automatically continuous \cite[Theorem 20]{PeRu}. A bounded conjugate linear operator $T : E \to X$ is said to be a \emph{local triple derivation} if for each $a\in E$, there exists a triple derivation $\delta_a : E \to X$ satisfying $T(a) = \delta_a (a)$. Clearly, every triple derivation is a local triple derivation, while the reciprocal implication is an open problem.

\begin{problem}\label{problem local triple derivations}
Is every local triple derivation on a JB$^*$-triple $E$ (or more generally, every local triple derivation from $E$ into a Jordan Banach triple $E$-module)
a triple derivation?
\end{problem}

In a Conference held in Hong-Kong in April 2012, M. Mackey announced a result establishing that, for each von Neumann algebra (and more generally, for every JBW$^*$-triple, i.e. a JB$^*$-triple which is also a dual Banach space), $W,$ every local triple derivation $T : W \to W$ is a triple derivation (see \cite[Theorem 5.11]{Mack}). Actually, the arguments provided by Mackey are also valid to prove that every local triple derivation on a compact JB$^*$-triple is a triple derivation. The proofs and techniques applied by M. Mackey in this result depend heavily on the particular structure of a JBW$^*$-triple and the abundance of tripotent elements in this setting. Mackey's theorem is an appropriate version of the aforementioned Kadison's theorem. The corresponding JB$^*$-triple version of Johnson's theorem is an open problem. Part of the above Problem \ref{problem local triple derivations} appears in \cite[Conjecture 6.2 (C1) and (C3)]{Mack}.\smallskip

Every C$^*$-algebra $B$ is a JB$^*$-triple with product $\J abc = \frac12\left( ab^* c+ cb^* a\right)$.
Triple and local triple derivations on $B$ make sense in this setting without any need to appeal to the above general
concepts on JB$^*$-triple setting.  The following C$^*$-version of the above Problem \ref{problem local triple derivations} is interesting by itself.

\begin{problem}\label{problem local triple derivations on Cstar}
Is every local triple derivation on a C$^*$-algebra $B$ a triple derivation?
\end{problem}

In this paper we focus on Problem \ref{problem local triple derivations on Cstar}.
Our main result shows that every local triple derivation on a unital C$^*$-algebra is a triple derivation (Theorem \ref{th Kadison thm for triple derivations on unital C*-algebras}). Section 3 contains a similar statement for local triple derivations on a unital JB$^*$-algebra. The results presented here connect local triple derivations on a unital C$^*$-algebra with generalised derivations, a concept studied by J. Li and Zh. Pan in \cite{LiPan}. We recall that a linear mapping $D$ from a unital C$^*$-algebra $A$ to a (unital) Banach $A$-bimodule $X$ is called a \emph{generalised derivation} whenever the identity $$D (ab) = D(a) b + a D(b) - a D(1) b$$ holds for every $a,b$ in $A$. We shall say that $D$ is a \emph{generalised Jordan derivation} whenever $D (a\circ b) = D(a)\circ b + a\circ D(b) - U_{a,b} D(1)$, for every $a,b$ in $A$, where the Jordan product is given by $a\circ b := \frac12 (a b + ba)$ and $U_{a,b} (x) := \frac12 ( axb + bxa) $.  Every generalised (Jordan) derivation $D : A\to X$  with $D(1) =0$ is a (Jordan) derivation. Let $A$ be a C$^*$-subalgebra of a C$^*$-algebra $B$. Suppose $B$ is unital and $A$ contains the unit, $1$, of $B$. The C$^*$-algebra $B$ can be regarded as $A$-bimodule with respect to its original product and as a (complex) Jordan Banach triple $A$-module with respect to $\J abc = \frac12\left( ab^* c+ cb^* a\right)$. In a first approach we prove that every (linear) local triple derivation from $A$ to $B$ is a generalised derivation. The main result establishes that every local triple derivation on a unital C$^*$-algebra is an associative derivation plus an inner triple derivation (see Theorem \ref{th Kadison thm for triple derivations on unital C*-algebras} and Corollary \ref{c local triple derivation are generalised Jordan derivation}).\smallskip

Although the proofs and results contained in this paper are developed only with techniques of C$^*$-algebra theory,
at some stage we have opted for a more general result in the setting of JB$^*$-triples and to pose Problem \ref{problem local triple derivations} in the more general context.

\section{Local triple derivations on unital C$^*$-algebras}

We recall that a \emph{JB$^*$-triple} is a complex Banach space $E$
equipped with a continuous triple product  $\J ... :
E\times E\times E \to E,$ which is symmetric and linear in the
first and third variables, conjugate linear in the second variable
and satisfies:
\begin{enumerate}[{\rm (a)}] \item The mapping $\delta (a,b) := L(a,b) - L(b,a)$ is a triple derivation on $E$,
where $L(a,b)$ is the operator on $E$ given by $L(a,b) x = \J abx;$
\item $L(a,a)$ is an hermitian operator with non-negative
spectrum; \item $\|L(a,a)\| = \|a\|^2$.\end{enumerate}

Every C$^*$-algebra is a JB$^*$-triple with respect $\J abc = \frac12 ( ab^* c+ cb^*a).$\smallskip

A {triple or ternary derivation}
$\delta$ on $E$ is said to be {\it inner} if it can be written as a finite sum of
derivations of the form $\delta(a,b)$ $(a,b\in E)$.\smallskip

Throughout this section, $A$ will denote a C$^*$-subalgebra of a unital C$^*$-algebra, $B$, containing the unit of $B$. 
The self-adjoint part of a C$^*$-algebra $\mathcal{B}$ will be denoted by $\mathcal{B}_{sa}.$
The C$^*$-algebra $B$ can be regarded as $A$-bimodule with respect to its original product and as a (complex) Jordan Banach triple $A$-module with respect to $\J abc = \frac12\left( ab^* c+ cb^* a\right)$. By an abuse of notation, a linear map $\delta : A \to B$ is called a triple derivation whenever it satisfies identity $(\ref{eq derivation})$ (beware that this is not exactly the definition introduced in \cite{HoPeRu}). A bounded linear operator $T: A\to B$ is a local triple derivation if for each $a$ in $A$ there exists a (linear) triple derivation $\delta_a : A \to B$ satisfying $\delta_a (a) =T(a).$

\begin{lemma}\label{l 1}\cite[Lemma 1]{HoMarPeRu} Let $T : A\to B$ be a local triple derivation. Then $T(1)^* = -T(1)$.
\end{lemma}

\begin{proof}
Take a triple derivation $\delta_{1} : A \to B$ satisfying $T(1) = \delta_1 (1)$. In this case, $$T(1) = \delta_1 \J 111= 2\J {\delta_1 (1)}11 + \J 1{\delta_1 (1)}1 $$ $$= 2 \delta_1 (1) + {\delta_1 (1)}^* = 2 T(1) + T(1)^*,$$ which implies that $T(1)^* = -T(1)$.
\end{proof}

The above lemma was also established in \cite[Proof of Lemma 1]{HoMarPeRu} and rediscovered in \cite[Lemma 3.1]{Mack}, the proof is included here for completeness reasons.\smallskip

We shall deduce now some consequences of the above Lemma \ref{l 1}. In the setting above, the mapping $\delta(T(1),1) = L(T(1),1) -L(1,T(1)) : A \to B$ is a triple or ternary derivation and $\delta(T(1),1) (1) = \J {T(1)}11- \J 1{T(1)}1 = T(1) - T(1)^* = 2 T(1)$. Thus, \begin{equation}
\label{eq local derivation annihilating zero}\widetilde{T}=T- \frac12 \delta(T(1),1) = T - \delta\left(\frac12 T(1),1\right)
\end{equation} is a local triple derivation and $\widetilde{T} (1) = T(1) - T(1) = 0$.\smallskip

We can exhibit now some examples of generalised derivations which are not local triple derivation.
A basic example is described as follows: let $a$ be an element in a C$^*$-algebra $B$,
the mapping $\hbox{adj}_{a} :B \to B$, $x\mapsto \hbox{adj}_{a} (x) :=a x - x a$, is an example of an associative derivation on $B$.
It is easy to see that the operator $G_{a} :B \to B$, $x\mapsto G_{a} (x):= a x + x a$, is an example of a generalised derivation on $B$.
Since, in the case of $B$ being unital, $G_{a} (1) = 2 a,$ it follows that $G_{a}$ is not a local ternary derivation whenever $a^* \neq -a$.\smallskip

The next lemma is established in the general setting of JB$^*$-triples although we shall 
only require the corresponding version for C$^*$-algebras.
Previously, we recall that elements $a,b$ in a JB$^*$-triple, $E,$ are
said to be \emph{orthogonal} ($a\perp b$ for short) if $L(a,b) =0$.
By Lemma 1 in \cite{BurFerGarMarPe} we know that $$a\perp b \Leftrightarrow \J aab = 0\Leftrightarrow\J bba=0.$$ 
When a C$^*$-algebra $\mathcal{B}$ is regarded as a JB$^*$-triple,
it is known that elements $a,b$ in $\mathcal{B}$ are orthogonal if, and only if, $a b^* = 0= b^* a=0$ (cf. \cite[\S 4]{BurFerGarPe}). 
When $\mathcal{B}$ is commutative, $a\perp b$ if, and only if, $ab=0$.

\begin{lemma}\label{l 2} Let $E$ be a JB$^*$-subtriple of a JB$^*$-triple $F$, where the latter is seen as a Jordan Banach triple $E$-module with respect to its natural triple product. Let $T: E \to F$ be a local triple derivation. Then the products of the form $\J a{T(b)}c$ vanish for every $a,b,c$ in $E$ with $a\perp b\perp c.$
\end{lemma}

\begin{proof} Let us take $a,b,c$ in $E$ satisfying $a\perp b\perp c,$ and a triple derivation $\delta_b : E\to F$ such that $\delta_b (b) = T(b)$. The identity $$\J a{T(b)}c = \J a{\delta_{b}(b)}c = \delta_{b} \J abc - \J {\delta_{b}(a)}bc - \J ab{\delta_{b}(c)} = 0,$$ proves the statement.
\end{proof}


It is due to B.E. Johnson that every bounded Jordan derivation from a
C$^*$-algebra $A$ to a Banach $A$-bimodule is an associative
derivation (cf. \cite{John96}). It is also known that every Jordan derivation
from a C$^*$-algebra $A$ to a Banach $A$-bimodule or to a Jordan Banach $A$-module
is continuous (cf. \cite[\S 1]{PeRu}). Therefore, every 
generalised Jordan derivation $D$ from a unital C$^*$-algebra $A$ to a Banach $A$-bimodule
with $D(1)=0$ is a bounded Jordan derivation and hence a
continuous associative derivation.\smallskip

We shall explore now the connections between generalised (Jordan) derivations and triple derivations from $A$ to $B$. Let $\delta: A \to B$ be a triple derivation. By Lemma \ref{l 1}, $\delta(1)^* = -\delta(1),$ and hence $$\delta(a\circ b) = \delta\J a1b = \J {\delta (a)}1b + \J a1{\delta (b)} + \J a{\delta (1)}b $$ $$= \delta(a) \circ b + a\circ \delta (b) + U_{a,b} (\delta(1)^*)= \delta(a) \circ b + a\circ \delta (b) - U_{a,b} (\delta(1)),$$ for every $a,b$ in $A$, which shows that $\delta$ is a generalised Jordan derivation (compare also \cite[Lemma 3.1]{Mack}).\smallskip

We shall focus now our attention on local triple derivations on a commutative unital C$^*$-algebra.

\begin{proposition}
\label{p l 3a} Let us assume that $A$ is commutative. Then every local triple derivation $T : A \to B$ 
is a generalised Jordan derivation.
\end{proposition}

\begin{proof}
Let us fix an arbitrary $\varphi \in B^*$ and define $W_{\varphi}: A\times A\times A \to \mathbb{C}$ a mapping given by $W_{\varphi} (a,b,c) := \varphi \left(\J {a}{T(b)}{c}\right)$. Clearly, $W_{\varphi}$ is linear and symmetric in $a$ and $c$ and conjugate linear in $b$. Lemma \ref{l 2} assures that \begin{equation}\label{eq p 1 3a 0} W_{\varphi} (a,b,c) ={\varphi} \J a{T(b)}c = \frac12 {\varphi} \left(a T(b)^* c + c T(b)^* a \right)=0,
\end{equation} whenever $a \perp b \perp c$ (or equivalently, $a b = b c= 0$). Fix $a,b\in A$ with $a b=0$. The form $V(s,t) := W_{\varphi}(a,bs,t)$ is linear in $t$ and conjugate linear in $s$ and $V(s,t) = 0$ for every $s,t\in A$ with $s t=0$. That is, $V$ an orthogonal form in the terminology of Goldstein in \cite{Gold}. It follows from \cite[Theorem 1.10]{Gold} (see also \cite{HaaLaust} or \cite{PaPeVi}) that there exists $\phi\in A^*$ satisfying \begin{equation}
\label{eq p 1 3a 1} V(s,t) = \phi (s^* t), \ \forall s,t\in A.
\end{equation} It follows from $(\ref{eq p 1 3a 1})$ that \begin{equation}
\label{eq p 1 3a 2} W_{\varphi}(a,b s,t)= V(s,t)= V(1,s^* t)= W_{\varphi}(a,b,s^* t)\end{equation}
for every $s,t,a,b\in A$ with $a b=0$. Fix $s,t\in A$, the above equation $(\ref{eq p 1 3a 2})$
shows that the form $V_2 (x,y) :=  W_{\varphi}(x,y s,t) - W_{\varphi}(x,y, s^* t)$ is orthogonal. Again Goldstein's theorem shows the existence of $\phi_1\in A^*$ satisfying $V_2 (x,y) = \phi_1 (x y^*)$, for every $x,y\in A$. Consequently, $V_2 (x,y) = V_2 (1, x^* y ) = V_2 (x y^*,1)$, for all $x,y\in A$. We have therefore proved that
$$ W_{\varphi}(x,y s,t) - W_{\varphi}(x,y, s^* t)  = W_{\varphi}(x y^*, s,t) - W_{\varphi}(x y^*,1, s^* t),$$ or equivalently,
$$ \varphi \left(\J {x}{T(ys)}{t} - \J {x}{T(y)}{s^* t} - \J {x y^*}{T(s)}{t} + \J {x y^*}{T(1)}{s^* t} \right)=0,$$ for every $x,y,s,t\in A$, $\varphi\in B^*$.
The arbitrariness of $\varphi$ and the Hahn-Banach theorem give \begin{equation}
\label{eq p 1 3a 3} \J {x}{T(ys)}{t} = \J {x}{T(y)}{s^* t} + \J {x y^*}{T(s)}{t} - \J {x y^*}{T(1)}{s^* t}.\end{equation} Finally, taking $x=t=1$, we have  $$ {T(ys)}^* = {T(y)}^* \circ {s^*} + y^* \circ {T(s)}^* - U_{y^*, s^*}\left({T(1)^*}\right),$$ which shows that $T$ is a generalised Jordan derivation.
\end{proof}

Let us make some observations to the proof of the above proposition. The identity $(\ref{eq p 1 3a 3})$ holds for every $x,y,s,t$ in $A$. Moreover, since, by Goldstine's theorem, the unit ball of $A$ is weak$^*$-dense in the unit ball of $A^{**}$, by Sakai's theorem, the products of $A^{**}$ and of $B^{**}$ are separately weak$^*$-continuous, and $T^{**}$ is weak$^*$-continuous, the equality \begin{equation}
\label{eq p 1 3a 4} \J {x}{T^{**}(ys)}{t} = \J {x}{T^{**}(y)}{s^* t} + \J {x y^*}{T^{**}(s)}{t} - \J {x y^*}{T(1)}{s^* t}.\end{equation} holds for every $x,y,s,t$ in $A^{**}$, and hence $T^{**}$ also is a generalised Jordan derivation.\smallskip

We can prove now a stronger version of Proposition \ref{p l 3a} which is a subtle variant of \cite[Sublemma 5]{Kad90} and \cite[Proposition 1.1]{LiPan}.

\begin{proposition}
\label{p l 3} In the hypothesis of Proposition \ref{p l 3a}, let $T : A \to B$ be a local triple derivation.
Then for each $a,b,c\in A$ with $ab= bc=0$ we have $$a T(b)^* c =a T(b^*)^* c =0.$$
\end{proposition}

\begin{proof} Fix $a,b,c\in A$ with $ab= bc=0$. Let us identify $A$ with some $C(K)$ for a suitable compact Hausdorff space $K$. Let $p$ denote the range projection of $b$ in $A^{**}$, that is $p = \chi_{_{S(b)}}$, where $S (b) :=\{ t\in K : b(t) \neq 0\}$ is the co-zero set of $b$. Observe that $a p=0=p c$ and $p b = b p =b$.\smallskip

By $(\ref{eq p 1 3a 4})$, we have $${(1-p)}{T(b)^*}{(1-p)} = \J {1-p}{T(b)}{1-p}  $$ $$= \J {1-p}{T(b p)}{1-p}= \J {1-p}{T(b)}{p(1-p)} + \J {(1-p) b^*}{T^{**}(p)}{1-p}  $$ $$ - \J {(1-p) b^*}{T(1)}{p (1-p)}=0.$$ Therefore, $a T(b)^* c = a (1-p) T(b)^* (1-p) c = 0.$
\end{proof}

One of the main results established by J. Li and Zh. Pan in \cite[Corollary 2.9]{LiPan} implies that a bounded linear operator $T : A \to B$ is a generalised derivation if, and only if, $a T(b) c= 0$, whenever $a b = b c=0$. Let us suppose that, in the above hypothesis, $A$ is commutative and $T : A \to B$ is a local triple derivation. Proposition \ref{p l 3} assures that $a T(b^*)^* c =0$, for every $a b= bc =0$ in $A$, and consequently, the mapping $x\mapsto T(x^*)^*$ is a generalised derivation, and thus, $$T(a^* b^*)^* = T(a^*)^* b + a T(b^*)^* - a T(1)^* b ,$$ or equivalently, $$ T(b a) =T(a b) = b T(a) + T(b) a - b T(1) a,$$ showing that $T$ is actually a generalised derivation. We have therefore proved the following:

\begin{corollary}\label{c local triple derivation are generalised Jordan derivation} Let us assume that $A$ is commutative. 
Then every local triple derivation from $A$ to $B$ is a generalised derivation. Moreover, taking $\widetilde{T}=T- \frac12 \delta(T(1),1) = T - \delta\left(\frac12 T(1),1\right),$ it follows that $\widetilde{T}$ is a local triple derivation with $\widetilde{T} (1) =0$, and hence $\widetilde{T}$ is a (Jordan) derivation.$\hfill\Box$
\end{corollary}

The statement concerning $\widetilde{T}$ in the above corollary could be also derived applying the
previously mentioned Johnson's theorem on the equivalence of Jordan derivations and (associative) derivations
(cf. \cite[Theorem 6.3]{John96}).\smallskip

\begin{remark}
\label{r GJder are Gder}{\rm  The argument given in the proof of Proposition \ref{p l 3}
is also valid to show that, under the same hypothesis, a generalised Jordan derivation $T: A \to B$
satisfies that $a T(b) c =0,$ for every $a,b,c\in A$ with $ab= bc=0$. Combining Goldstine's theorem 
with the separate weak$^*$-continuity of the product of $A^{**}$ and $B^{**}$ we guarantee 
that $T^{**}$ is a generalised Jordan derivation too. Let $p$ denote the range projection of $b$ in $A^{**}$.
In this case $$T(b)= T(p\circ b) = T(p) \circ b + p \circ T(b) - U_{p,b} T(1)$$ $$=\frac12 \Big(b T(p) + T(p) b + p T(b) + T(b) p - p T(1) b - b T(1)p \Big),$$ which implies that $(1-p) T(b) (1-p)=0,$ and hence $a T(b) c =0,$ for every $a,b,c\in A$ with $ab= bc=0$.
By \cite[Corollary 2.9]{LiPan}, $T$ is a generalised derivation. This shows that every generalised Jordan derivation
on a unital C$^*$-algebra is a generalised derivation. }
\end{remark}

Associative derivations from $A$ to $B$ are not far away from triple derivation. It is not hard to check that, in our setting, a bounded linear operator $\delta : A \to B$ is a triple derivation and $\delta(1)=0$ if, and only if, it is a $^*$-derivation, that is, it is a derivation and $\delta(a^*)= \delta(a)^*$.\smallskip

\begin{lemma}\label{l local triple derivation is symmetric} Let $B$ be a unital C$^*$-algebra,
and let $T : B\to B$ be a bounded local triple derivation with $T(1) = 0$.
Then $T$ is a symmetric operator, that is, $T(a^*) = T(a)^*$, for every $a\in B$.
\end{lemma}

\begin{proof} Let $A$ denote the abelian C$^*$-subalgebra generated by a normal element $a$ and the unit of $B$. Since $T|_{A} : A\to B$ is a local triple derivation and $T(1)=0$, by Corollary \ref{c local triple derivation are generalised Jordan derivation} and the subsequent comments, $T|_{A}$ is an associative derivation. Let $u$ be a unitary element in $A$. Since $T|_{A}$ is a derivation, we have $0=T(1)=T(u u^*)=u T(u^{*})+T(u) u^{*},$ so $$T(u)=-u T(u^{*}) u.$$

Now, having in mind that $T$ is a local triple derivation, there exists a triple derivation $\delta_u$ such that $T(u)=\delta_u (u)$,
we deduce that $T(u)=\delta_u (u) = \delta_{u} (u u^{*} u)= \delta_{u} \{ u,u,u \}=2 \{u,u, T (u) \}+ \{u,T(u),u \}= 2 T(u)+ u T(u)^* u,$
which gives $$T(u)= - u  T(u)^* u.$$ Combining these two equations we have $u T(u^*) u = u T(u)^* u,$ and hence $T(u^*)=T(u)^*$.\smallskip

Finally, by the Russo-Dye theorem, $T(b^*) = T(b)^*$, for every $b$ in $A$.
The arbitrariness of the normal element $a$ implies that $T(b)^* = T(b)$, for every $b\in B_{sa}$,
which gives the statement of the lemma.
\end{proof}

We can state now the main result.

\begin{theorem}\label{th Kadison thm for triple derivations on unital C*-algebras}
Let $B$ be a unital C$^*$-algebra. Every local triple derivation from $B$ to $B$ is a triple derivation.
\end{theorem}

\begin{proof} 
Let $A$ denote the abelian C$^*$-subalgebra generated by a normal element $a$ and the unit of $B$. Since $T|_{A} : A\to B$ is a local triple derivation, we can apply Corollary \ref{c local triple derivation are generalised Jordan derivation} and the comments following it to assure that $T|_{A}$ is a triple derivation and $\widetilde{T}|_{A}=\left(T- \frac12 \delta(T(1),1)\right)|_{A}$ is an associative derivation. It follows that $\widetilde{T} (a^2) = \widetilde{T} (a)  a + a \widetilde{T} (a)$. Since $a$ was arbitrarily chosen, we can affirm that $$\widetilde{T} ((a+b)^2) = \widetilde{T} (a+b)  (a+b) + (a+b) \widetilde{T} (a+b),$$ for every $a,b\in B_{sa}$, which implies that $$\widetilde{T} (a\circ b) = \widetilde{T} (a) \circ b + a \circ \widetilde{T} (b),$$ for every $a,b\in A_{sa}$. It is easy to check that $\widetilde{T}$ is a Jordan derivation, and hence an associative derivation by \cite[Theorem 6.3]{John96}. Now, Lemma \ref{l local triple derivation is symmetric} assures that $\widetilde{T}$ is a symmetric operator
and thus a triple derivation, which concludes the proof.
\end{proof}

We shall conclude this section with a result on ``automatic continuity'' for generalised derivations.
The following construction is inspired by the results in \cite[\S 4]{GarPe} (see also \cite{PeRu}). 
Let $D: B\to X$ be a generalised Jordan derivations from a unital C$^*$-algebra to a Banach $B$-module.
We regard $X$ as a Jordan Banach triple $B$-module with the triple products defined by $\J xba= \J abx := (a\circ b) \circ x
(x\circ b) \circ a- (a\circ x) \circ b$, and $\J axb := (a\circ x) \circ b
(x\circ b) \circ a- (a\circ b) \circ x,$ where for each $a\in B$ and $x\in X$,  $a\circ x := \frac12 (a x + x a)$.  
 
Fix $a,b,c\in B_{sa}$. The identity 
$$ D(\{a,b,c\}) - \{ D(a),b ,c\} -\{a ,D(b) ,c\}-\{a ,b ,D(c)\} =$$
$$= - U_{a,b}(D(1))\circ c - U_{a\circ b, c}(D(1)) - U_{c,b}(D(1))\circ a - U_{c\circ b, a}(D(1)) $$
$$+  U_{a,c}(D(1))\circ b + U_{c\circ a, b}(D(1)),$$ shows that the mapping 
$\check{D}|_{B_{sa}^{3}} : B_{sa}\times B_{sa}\times B_{sa} \to X$, 
$\check{D} (a,b,c) := D (\J abc) - \J {D(a)}bc -\J a{D(b)}c - \J ab{D(c)}$ is a continuous trilinear operator 
and hence $D$ is a ``\emph{generalised triple derivation}'' in the terminology employed in \cite[\S 4]{GarPe}.
It follows from \cite[Proposition 21]{GarPe} (see also \cite[Theorem 22]{GarPe}) that 
$D|_{B_{sa}}$ is continuous. The continuity of $D$ follows straightforwardly.\smallskip

\begin{proposition}\label{p automatic continuity of generalised derivations} Every generalised (Jordan) derivation, 
not assumed a priori to be continuous, from a unital C$^*$-algebra $B$ into a Banach $B$-bimodule is continuous.$\hfill\Box$
\end{proposition}

Despite the automatic continuity of generalised derivations, in the results included in this section,
local triple derivations, generalised derivations and generalised Jordan derivations
are assumed to be continuous, and these assumptions are needed in the arguments. 
The results established by J. Li and Zh. Pan in \cite[Proposition 1.1 and Corollary 2.9]{LiPan}
on generalised derivations need to assume hypothesis of continuity. 

\begin{problem}\label{problem continuity local triple derivations}\cite[Conjecture 6.2 (C2)]{Mack}
Is every local triple derivation, not assumed a priori to be continuous, on a C$^*$-algebra or on a JB$^*$-triple $E$ continuous?
\end{problem}

\section{Local triple derivations on unital JB$^*$-algebras}

Every JB$^*$-algebra $J$ can be equipped with a structure of JB$^*$-triple with respect to the product
$$\J abc := (a \circ b^*) \circ c + (c \circ b^*) \circ a - (a \circ c) \circ b^*.$$
A Jordan derivation on $J$ is a linear mapping $d : J\to J$ satisfying $d(a\circ b) = d(a)\circ b + a\circ d(b)$,
for every $a,b\in J$. Given a Jordan-Banach triple $J$-module $X$, a conjugate linear mapping $\delta: J \to X$
is said to be a triple derivation whenever the identity $$\delta \J abc = \J {\delta (a)}bc + \J a{\delta (b)}c+ \J ab{\delta (c)},$$
holds for every $a,b,c\in J$.\smallskip

According to what we did in the setting of C$^*$-algebras, given a unital JB$^*$-algebra $J$ and a JB$^*$-subalgebra, $A$,
containing the unit of $J$, $J$ can be regarded as a Jordan-Banach $J$-module and a Jordan-Banach triple $A$-module with respect to
its natural Jordan product and its natural triple product, respectively. By a little abuse of notation, a linear mapping $\delta : A\to J$
satisfying $\delta \J abc = \J {\delta (a)}bc + \J a{\delta (b)}c+ \J ab{\delta (c)},$ for every $a,b,c\in A,$ is said to be a \emph{triple derivation}.
A \emph{local triple derivation} from $A$ to $J$ is bounded linear operator $T : A \to B$ such that for each $a\in A$ there exists a triple derivation $\delta_{a} : A \to J$ satisfying $T(a) = \delta_a (a)$.\smallskip

Arguing as in the previous section, we have:

\begin{lemma}\label{l 1 b}\cite[Lemma 1]{HoMarPeRu} Let $A$ be a JB$^*$-subalgebra of a unital JB$^*$-algebra $J$
containing the unit of $J$, and let $T : A\to J$ be a local triple derivation.
Then $T(1)^* = -T(1)$.$\hfill\Box$
\end{lemma}

As in the C$^*$-setting, the mapping $\delta(T(1),1) = L(T(1),1) -L(1,T(1)) : A \to J$
is an inner triple or ternary derivation, $\delta(T(1),1) (1) = 2 T(1)$, and $\widetilde{T}=T- \frac12 \delta(T(1),1)$
is a local triple derivation with $\widetilde{T} (1) = 0$.\smallskip

Motivated by the definitions made in the associative setting, a linear mapping $D:A\to J$
is a \emph{generalised Jordan derivation} whenever $D (a\circ b) = D(a)\circ b + a\circ D(b) - U_{a,b} D(1)$,
for every $a,b$ in $A$. Every generalised Jordan derivation $D : A\to J$ with $D(1) =0$ is a Jordan derivation
and every triple derivation $\delta : A\to J$ is a generalised Jordan derivation.\smallskip

The proof of Proposition \ref{p l 3a} remains valid in the following sense:

\begin{proposition}
\label{p l 3a b} Let $A$ be the (associative) JB$^*$-subalgebra of
a unital JB$^*$-algebra $J$ generated by a self-adjoint element $a$ and the unit of $J$.
Suppose $T : A \to J$ is a local triple derivation, then $T$ is a generalised Jordan derivation.$\hfill\Box$
\end{proposition}

Since the proof of Lemma \ref{l local triple derivation is symmetric} remains valid in the Jordan setting, the reasoning
given in Corollary \ref{c local triple derivation are generalised Jordan derivation}
and Theorem \ref{th Kadison thm for triple derivations on unital C*-algebras} can be rephrased to prove the following:

\begin{theorem}\label{th Kadison thm for triple derivations on unital JB*-algebras}
Let $J$ be a unital JB$^*$-algebra. Every local triple derivation from $J$ to $J$ is a triple derivation.$\hfill\Box$
\end{theorem}

\bigskip\bigskip


\begin{thebibliography}{10}






\bibitem{BarFri} T.J. Barton, Y. Friedman, Bounded derivations of JB$^*$-triples,
\emph{Quart. J. Math. Oxford} \textbf{41}, 255-268 (1990).










\bibitem{BurFerGarMarPe} M. Burgos, F.J. Fern{\'a}ndez-Polo, J. Garc{\'e}s, J.
Mart{\'\i}nez, A.M. Peralta, Orthogonality preservers in
C$^*$-algebras, JB$^*$-algebras and JB$^*$-triples, \emph{J. Math. Anal.
Appl.} \textbf{348}, 220-233 (2008).

\bibitem{BurFerGarPe} M. Burgos, F.J. Fern{\' a}ndez-Polo, J.J.
Garc{\'e}s, and A.M. Peralta, Orthogonality preservers Revisited,
\emph{Asian-European Journal of Mathematics} \textbf{2}, No. 3,
387-405 (2009).









\bibitem{Crist} R. Crist, Local derivations on operator algebras, \emph{J. Funct. Anal.} \textbf{135}, 72-92 (1996).














\bibitem{GarPe} J.J. Garc{\'e}s, A.M. Peralta, Generalised triple homomorphisms and Derivations, to appear in \emph{Canadian J. Math.}

\bibitem{Gold} S. Goldstein, Stationarity of operator algebras,
\emph{J. Funct. Anal.} \textbf{118}, no. 2, 275-308 (1993).


\bibitem{HaaLaust} U. Haagerup, N.J. Laustsen, Weak amenability of C$^*$-algebras and a theorem of Goldstein, in  \emph{Banach algebras '97 (Blaubeuren)}, 223-243, de Gruyter, Berlin, 1998.

\bibitem{HadLi} D. Hadwin, J. Li, Local derivations and local automorphisms, \emph{J. Math. Anal. Appl.} \textbf{290}, 702-714 (2004).

\bibitem{HadLi08} D. Hadwin, J. Li, Local derivations and local automorphisms on some algebras, \emph{J. Operator Theory} \textbf{60}, 29-44 (2008).

\bibitem{Hv} B. Hvala, Generalized derivations in rings, \emph{Comm. Algebra} \textbf{26}, 1147-1166 (1998).








\bibitem{HoMarPeRu} T. Ho, J. Martinez-Moreno, A.M. Peralta, B. Russo,
 Derivations on real and complex JB$^\ast$-triples, \emph{J. London Math. Soc.} (2)
 \textbf{65}, no. 1, 85-102 (2002).


\bibitem{HoPeRu} T. Ho, A.M. Peralta, B. Russo, Ternary Weakly Amenable C$^*$-algebras and JB$^*$-triples,
to appear in \emph{Quart. J. Math. (Oxford)}.








\bibitem{John87} B.E. Johnson, Continuity of generalized homomorphisms,
\emph{Bull. London Math. Soc.} \textbf{19},  no. 1, 67-71 (1987).

\bibitem{John96} B.E. Johnson, Symmetric amenability and the nonexistence of Lie and
Jordan derivations, \emph{Math. Proc. Cambridge Philos. Soc.}
\textbf{120}, no. 3, 455-473 (1996).

\bibitem{John01} B.E. Johnson, Local derivations on C$^*$-algebras are derivations,
\emph{Trans. Amer. Math. Soc.} \textbf{353}, 313-325 (2001).




\bibitem{Kad90} R. Kadison, Local derivations, \emph{J. Algebra} \textbf{130}, no. 2, 494–509 (1990).




\bibitem{Ka} W. Kaup, A Riemann Mapping Theorem for bounded symmentric domains in complex Banach spaces, \emph{Math. Z.} \textbf{183}, 503-529 (1983).









\bibitem{LarSo} D. Larson, A. Sourour, Local derivations and local automorphisms of $B(X)$, \emph{Proc. Sympos. Pure Math.} \textbf{51}, 187-194  (1990).

\bibitem{LiPan} J. Li, Zh. Pan, Annihilator-preserving maps, multipliers, and derivations, \emph{Linear Algebra Appl.}
\textbf{423}, 5-13 (2010).

\bibitem{LiPanXu} J. Li, Z. Pan, H. Xu, Characterizations of isomorphisms and derivations of some algebras,
\emph{J. Math. Anal. Appl.} \textbf{332}, 1314-1322 (2007).




\bibitem{Mack} M. Mackey, Local derivations on Jordan triples, preprint 2012 (arXiv:1207.5394v1).







\bibitem{PaPeVi} C. Palazuelos, A.M. Peralta, I. Villanueva, Orthogonally Additive Polynomials on C$^*$-algebras,
\emph{Quart. J. Math. Oxford.} \textbf{59}, (3), 363-374 (2008).




\bibitem{PeRu} A.M. Peralta, B. Russo, Automatic continuity of derivations on
C$^*$-algebras and JB$^*$-triples, preprint 2010.







\bibitem{Ringrose72} J.R. Ringrose, {Automatic continuity of derivations of operator algebras},
\emph{J. London Math. Soc.} (2) \textbf{5} , 432-438 (1972).







\bibitem{Shu} V. Shulman, Operators preserving ideals in C$^*$-algebras,
\emph{Studia Math.} \textbf{109}, 67-72 (1994).















\bibitem{ZhanPanYang} J. Zhang, F. Pan, A. Yang, Local derivations on certain CSL algebras,
\emph{Linear Algebra Appl.} \textbf{413}, 93-99 (2006).

\end{thebibliography}
\end{document}